\documentclass[12pt,a4paper,draft]{amsart}
\usepackage{amsfonts}
\usepackage{amssymb}
\usepackage{amsmath}
\usepackage{amsthm}
\usepackage{cite}
\usepackage{graphicx}
\usepackage{amscd}
\usepackage{color}
\newtheorem{theorem}{Theorem}
\theoremstyle{plain}

\newtheorem{corollary}[theorem]{Corollary}

\newtheorem{lemma}[theorem]{Lemma}

\newtheorem{proposition}[theorem]{Proposition}
\newtheorem{remark}[theorem]{Remark}
\numberwithin{equation}{section}

\def\eps{\varepsilon}

\newcommand{\R}{\mathbb{R}}

\newcommand{\mS}{\mathbb{S}}
\newcommand{\N}{\mathbb{N}}

\newcommand{\B}{{\bf B}}

\newcommand{\cC}{{\mathcal C}}

\newcommand{\cM}{{\mathcal M}}

\renewcommand{\phi}{\varphi}

\DeclareMathOperator{\divergenz}{div}
\DeclareMathOperator{\const}{const.}
\DeclareMathOperator{\gr}{gr}

\newcommand{\dist}{\text{\rm dist}}

\newenvironment{altproof}[1]
{\noindent%\addvspace{0.3cm}
{\bf Proof of {#1}}.}
{\nopagebreak\mbox{}\hfill $\Box$\par\addvspace{0.5cm}}

\begin{document}

\title[A priori bounds for higher-order elliptic problems]{A priori bounds and 
a Liouville theorem on a half-space for higher-order elliptic 
Dirichlet problems} 

\author{Wofgang Reichel}
\address{W. Reichel \hfill\break 
Mathematisches Institut, Universit\"at Giessen \hfill\break
Arndtstr. 2, D-35392 Giessen, Germany}
\email{wolfgang.reichel@math.uni-giessen.de}

\author{Tobias Weth}
\address{T. Weth \hfill\break 
Mathematisches Institut, Universit\"at Giessen \hfill\break
Arndtstr. 2, D-35392 Giessen, Germany}
\email{tobias.weth@math.uni-giessen.de}
\date{\today}


\begin{abstract} 
We consider the $2m$-th order elliptic boundary value problem 
$Lu=f(x,u)$ on a bounded smooth domain $\Omega\subset\R^N$ with Dirichlet boundary conditions $u=
\frac{\partial}{\partial\nu}u= \ldots=(\frac{\partial}{\partial\nu})^{m-1} u =0$ on 
$\partial\Omega$. The operator $L$ is a uniformly elliptic operator of order 
$2m$ given by $L=\big(-\sum_{i,j=1}^N a_{ij}(x) 
\frac{\partial^2}{\partial x_i\partial x_j}\big)^m 
+\sum_{|\alpha|\leq 2m-1} b_\alpha(x) D^\alpha$. 
For the nonlinearity we assume that $\lim_{s\to\infty}\frac{f(x,s)}{s^q}=h(x)$, 
$\lim_{s\to-\infty}\frac{f(x,s)}{|s|^q}=k(x)$ where $h,k\in C(\overline{\Omega})$ are positive 
functions and $q>1$ if $N\leq 2m$, $1<q<\frac{N+2m}{N-2m}$ if $N>2m$.
We prove a priori bounds, i.e, we show that 
$\|u\|_{L^\infty(\Omega)} \leq C$ for every solution $u$, where $C>0$ 
is a constant. The solutions are allowed to be sign-changing. The 
proof is done by a blow-up argument which relies on the following new Liouville-type theorem on a 
half-space: if $u$ is a classical, bounded, non-negative solution of 
$(-\Delta)^m u = u^q$ in $\R^N_+$ with 
Dirichlet boundary conditions on $\partial\R^N_+$ and $q>1$ if $N\leq 2m$, 
$1<q\leq\frac{N+2m}{N-2m}$ if $N>2m$ then $u\equiv 0$.  
\end{abstract}

\maketitle

%%%%%%%%%%%%%%%%%%%%%%%%%%%%%%%%%%%%%%%%%%%%%%%%%%%%%%%%%%%%%%%%%%%%%%%%%%%

% Let $\Omega\subset\R^N$ be a domain with $\partial\Omega\in C^{2m}, m\in \N$. 

\section{Introduction}
A priori bounds for solutions of elliptic boundary value problems have been of major importance 
at least as far back as Schauder's work in the 1930s. In this paper we
prove a priori estimates 
on bounded, smooth domains $\Omega\subset\R^N$ for solutions of higher order boundary value 
problems of the form 
\begin{equation}
L u = f(x,u) \mbox{ in } \Omega, \quad u=\frac{\partial}{\partial\nu}u=
\ldots=\left(\frac{\partial}{\partial\nu}\right)^{m-1} u =0 \mbox{ on } 
\partial\Omega.
\label{basic}
\end{equation}
Here $\nu$ is the unit exterior normal on $\partial\Omega$ and 
$$
L = \Bigl(- \sum_{i,j=1}^N a_{ij}(x) 
\frac{\partial^2}{\partial x_i\partial x_j}\Bigr)^m 
+\sum_{|\alpha|\leq 2m-1} b_\alpha(x) D^\alpha 
$$
is a uniformly elliptic operator with coefficients $b_\alpha\in L^\infty(\Omega)$ and 
$a_{ij}\in C^{2m-2}(\overline{\Omega})$ such that there exists a 
constant $\lambda>0$ with $\lambda^{-1} |\xi|^2 \leq \sum_{i,j=1}^N a_{ij}(x)\xi_i\xi_j \leq 
\lambda |\xi|^2$ for all $\xi\in \R^N$, $x\in \Omega$. Our main result is the following:

\begin{theorem}
Suppose $\Omega\subset\R^N$ is a bounded domain with $\partial\Omega\in 
C^{2m}$. Let $m\in \N$ and assume that $q>1$ if $N\leq 2m$ and $1<q<\frac{N+2m}{N-2m}$ if $N>2m$. 
Suppose further that there exist positive, continuous 
functions $k,h: \overline{\Omega}\to (0,\infty)$ such that 
\begin{equation}
  \label{eq:asymptotic}
\lim_{s\to +\infty} \frac{f(x,s)}{s^q} = h(x), \quad 
\lim_{s\to-\infty} \frac{f(x,s)}{|s|^q}= k(x)
\end{equation}
uniformly with respect to $x\in \overline{\Omega}$. 
Then there exists a constant $C>0$ depending only on the data 
$a_{ij}, b_\alpha, \Omega, N, q, h, k$ such that $\|u\|_\infty \leq C$ for 
every solution $u$ of \eqref{basic}. 
\label{main}
\end{theorem}

\begin{remark} 
Suppose the nonlinearity depends on a real parameter $\lambda$, i.e, 
$f_\lambda: \Omega\times \R\to \R$ and 
$$
\lim_{s\to +\infty} \frac{f_\lambda(x,s)}{\lambda s^q} = h(x), \quad 
\lim_{s\to -\infty} \frac{f_\lambda(x,s)}{\lambda|s|^q}= k(x)
$$
uniformly with respect to $x\in \Omega$ and $\lambda\in [\lambda_0,\infty)$ where $\lambda_0>0$. 
Then the a priori bound of Theorem \ref{main} depends additionally on $\lambda_0$ but not 
on $\lambda$. This is important in the study of global solution branches of a parameter dependent 
version of \eqref{basic}, which we will pursue in future work.
\label{remark_to_main}
\end{remark}

We are focusing on the case of superlinear nonlinearities $f(x,u)$ with subcritical 
growth. A model nonlinearity is $f(x,s)=|s|^q$. Our results hold with no restriction on the 
shape of the domain $\Omega$ and for general, possibly sign-changing solutions. This is important 
since the lack of the maximum principle for higher order equations does not allow to restrict 
attention to positive solutions only. 

In the second-order case $m=1$ a priori bounds for positive solutions
have been 
established for subcritical, superlinear nonlinearities via different methods by 
Brezis, Turner \cite{BT}, Gidas, Spruck \cite{GS1}, 
DeFigueiredo, Lions and Nussbaum \cite{DLN} and recently by Quittner, Souplet \cite{QS} and 
McKenna, Reichel \cite{MCR}. In the higher-order 
case $m\geq 2$ the theory is far less developed and strongly depends
on the type of boundary conditions considered. For Dirichlet boundary 
conditions we only know of a result of Soranzo \cite{SO}, who 
proved a priori bounds for positive radial solutions on a ball if $L=(-\Delta)^m$. 
For Navier boundary conditions the picture is more complete. Let $L= (-L_0)^m$ 
where $L_0=a_{ij}\frac{\partial^2}{\partial x_i\partial x_j} + 
b_\alpha \frac{\partial}{\partial x_\alpha}$ is a 
second order operator and suppose the boundary conditions are of Navier-type: 
\begin{equation}
  \label{eq:navier}
u = (-L)u = \ldots =(-L)^{m-1} u=0 \mbox{ on } \partial \Omega.
\end{equation}
Soranzo \cite{SO} proved a priori bounds for 
positive solutions if $L_0=\Delta$ and $\Omega$ is a bounded smooth convex domain. Recently, 
Sirakov \cite{SI} improved this result to general operators $L=(-L_0)^m$ and general bounded 
smooth domains. Both authors strongly use the fact that the boundary
conditions \eqref{eq:navier} allow to write the problem as a coupled
system of second order equations, where each equation is complemented 
with Dirichlet boundary conditions. In this case maximum principles
are available. In contrast, the higher order Dirichlet problem can not
be rewritten as a system and therefore requires different techniques.

In our approach we extend the so-called ``scaling argument'' of Gidas and 
Spruck \cite{GS1}, which they used to deal with the second order case
$m=1$ and positive solutions. Let us give a brief sketch of their 
method. Gidas and Spruck assume that there exists a sequence of positive solutions with 
$L^\infty$-norm tending to $+\infty$. After rescaling the 
solutions to norm $1$ and blowing-up the coordinates one can take a limit of the rescaled 
solutions and obtains a nontrivial positive solution of a limit boundary value
problem $-\Delta u = u^q$ 
on either the full-space $\R^N$ or the half-space $\R^N_+=\{x\in \R^N:x_1>0\}$ together with 
Dirichlet boundary conditions. Then a contradiction is reached
provided that a 
Liouville-type result is available, i.e., a result which shows that the non-negative solutions of 
the limit problem must be identically zero. For subcritical $q$ 
Gidas and Spruck \cite{GS1}, \cite{GS2} proved 
both the full-space and the half-space Liouville theorem for $-\Delta
u = u^q$ via the method of moving planes.

In order to deal with the higher order Dirichlet problem \eqref{basic} and
solutions which may change sign, the blow up procedure has to be
modified. Indeed, even under assumption \eqref{eq:asymptotic}, there
seems to be no direct argument to exclude the
case of {\em negative blow up} (i.e., the existence of a sequence of
solutions which is not uniformly bounded from below). Instead, it is excluded a
posteriori after passing to the limit equation. Once this is done, we
still need Liouville theorems for nonnegative solutions of the higher 
order problems on $\R^N$, 
$\R^N_+$. The full-space Liouville theorem stated next is already known; it was proved 
by Lin \cite{Lin} if $m=2$ and for general $m\geq 2$ by Wei, Xu \cite{WeiXu}. 

\begin{theorem}[Wei, Xu] Let $m \in \N$ and assume that $q>1$ if $N\leq 2m$ and 
$1<q<\frac{N+2m}{N-2m}$ if $N>2m$. If $u$ is a classical non-negative solution of 
$$
(-\Delta)^m u = u^q \mbox{ in } \R^N,
$$ 
then $u\equiv 0$.
\label{rn}
\end{theorem}

Even in the case of the Navier boundary conditions, the corresponding Liouville 
theorem for the polyharmonic problem in the half-space is harder to achieve and has only 
recently been proved by Sirakov~\cite{SI}. Due to the lack of a (local) maximum principle, the
corresponding Dirichlet problem is even more difficult to deal
with. Here we show the following new Liouville Theorem for the half-space 
which complements Theorem~\ref{rn}.

\begin{theorem} Let $m \in \N$ and assume that $q>1$ if $N\leq 2m$ and $1<q\leq \frac{N+2m}{N-2m}$ 
if $N>2m$. If $u$ is a classical non-negative bounded 
solution of 
\begin{equation}
  \label{eq:liouville-half-space}
(-\Delta)^m u = u^q \mbox{ in } \R^N_+, \quad u=
\frac{\partial}{\partial x_1}u=\ldots=
\frac{\partial^{m-1}}{\partial x_1^{m-1}} u =0 \mbox{ on }\partial \R^N_+
\end{equation}
then $u\equiv 0$.
\label{rn_plus}
\end{theorem} 
We point out that the critical case $q=\frac{N+2m}{N-2m}$ is allowed
in Theorem~\ref{rn_plus}. Note also that Theorem~\ref{rn_plus} holds in the class of bounded 
solutions. It remains an open problem to extend the result to the class of all 
(possibly unbounded) classical positive solutions.

Let us outline the proof of Theorem~\ref{rn_plus} and point out the
main difficulties. Following Gidas and Spruck, 
we transform the half-space problem via a Kelvin inversion into a
problem in 
the unit ball, where the point at infinity is mapped onto the boundary point 
$P=(-1,0,\ldots,0)$. The transformed solution satisfies
\begin{equation}
 \label{after_kelvin}
\begin{array}{ll}
(-\Delta)^m v = 2^{2m}|x-P|^{-\alpha} v^q & \mbox{ pointwise in } B_1(0), \vspace{\jot}\\
\displaystyle v=\frac{\partial}{\partial\nu}v=
\ldots=\left(\frac{\partial}{\partial\nu}\right)^{m-1} v =0 & \mbox{ on } 
\partial B_1(0)\setminus\{P\},
\end{array}
\end{equation}
where $\alpha=N+2m-q(N-2m)\geq 0$. The key step is to
show that $v$ is axially symmetric around the $x_1$-axis. In the second
order case $m=1$ this is proved with the classical moving plane method, which is a local 
method based on the maximum principle. The same local approach fails for the higher order
case $m\geq 2$ since the maximum principle is not available. Very recently, a new development 
in the moving plane procedure by Berchio, Gazzola and Weth \cite{BGW} overcame part of this 
difficulty.  The authors made the moving plane method applicable 
for classical solutions of polyharmonic Dirichlet problems on balls. 
Instead of a local maximum principle method they argue via the Green integral-representation and 
properties of the Green function. However, the method of \cite{BGW} does not
apply here, since the solution $v$ of \eqref{after_kelvin} may have a singularity at 
$P\in\partial B_1(0)$. To overcome this problem, a large part of this
work is devoted to show
that every solution $v$ of \eqref{after_kelvin} which corresponds to a bounded solution of
\eqref{eq:liouville-half-space} can be of represented via the Green
function. In this step we also use Green function estimates of Grunau
and Sweers \cite{GrSw}. Then we apply a moving plane argument, using the Green
function representation and the Hardy-Littlewood-Sobolev inequality,
to get the desired symmetry result. Comparing this variant of
the moving plane method with the one in \cite{BGW}, we point out that
Berchio, Gazzola and Weth allow more general (non-Lipschitz)
nonlinearities, but their argument relies on Green function representations
for directional derivatives of the solution which in our situation might not exist.

Once the symmetry result for $v$ is established, we readily conclude
-- following Gidas and Spruck \cite{GS1},\cite{GS2} again -- that the
corresponding solution $u$ of \eqref{eq:liouville-half-space} is axially symmetric around 
{\em any} axis parallel to the $x_1$-axis. Consequently, $u$ is a function of $x_1$ only and 
hence solves an ordinary differential equation. It is then easy to conclude that $u\equiv 0$. 

We recall that the original moving plane method goes back
to Alexandrov \cite{A} and Serrin \cite{S} 
and was further developed by Gidas, Ni, Nirenberg \cite{GNN} for second order equations. Recent 
improvements of the moving plane method for higher order equations and pseudo differential 
operators using integral representations rather than local 
maximum principles were achieved by Chang, Yang \cite{Chang_Yang}, Berchio, Gazzola, 
Weth \cite{BGW}, Li \cite{yyli}, Chen, Li, Ou \cite{Chen_Li_Ou} and 
Birkner, L\'{o}pez-Mimbela, Wakolbinger \cite{BiLoWa}.

The paper is organized as follows. In Section~\ref{sec:blow_up} we prove Theorem~\ref{main} 
assuming Theorem~\ref{rn_plus}. We give the details of the blow-up procedure taking into account 
that we allow for solutions blowing up to either at $+\infty$ or $-\infty$. The rest of the 
paper is devoted to the proof of Theorem~\ref{rn_plus}. In Section~\ref{sec:representation} we 
prove a Green-representation formula on half-spaces (Theorem~\ref{green-poiss-form_plus}) by 
approximating the half-space by a family of 
growing balls. Based on the Green-representation for balls and by a careful estimate of the 
boundary integrals and the monotone convergence theorem we obtain a Green-representation for the 
half-space. Finally, in Section~\ref{sec:proof-theorem} we prove
Theorem~\ref{rn_plus}.

%%%%%%%%%%%%%%%%%%%%%%%%%%%%%%%%%%%%%%%%%%%%%%%%%%%%%%%%%%%%%%%%%%%%%%%%%%%%%%
\section{Proof of Theorem \ref{main} -- the blow-up argument} \label{sec:blow_up}

In this section we give the details of the blow-up argument for the proof of Theorem~\ref{main} 
under the assumption of the validity of the Liouville-type result of Theorem~\ref{rn_plus}. The 
proof uses standard linear $L^p$-$W^{2m,p}$ estimates for linear problems
\begin{align}
L u &= g(x) \mbox{ in } \Omega, \label{linear_eq}\\
u &=\frac{\partial}{\partial\nu}u=
\ldots=\left(\frac{\partial}{\partial\nu}\right)^{m-1} u =0 \mbox{ on } 
\partial\Omega. \label{linear_bc}
\end{align}
Recall the following basic estimate of Agmon, Douglis, Nirenberg \cite{ADN}.

\begin{theorem}[Agmon, Douglis, Nirenberg] Let 
$\Omega\subset\R^N$ be a bounded domain with $\partial\Omega\in C^{2m}, 
m\in \N$. Let $a_{ij}\in C^{2m-2}(\overline{\Omega})$, 
$b_\alpha \in L^\infty(\Omega)$, $g\in L^p(\Omega)$ for some 
$p\in (1,\infty)$. Suppose $u\in W^{2m,p}(\Omega)\cap W_0^{m,p}(\Omega)$ 
satisfies \eqref{linear_eq}. Then there exists a constant 
$C>0$ depending only on $\|a_{ij}\|_{C^{2m-2}}, \|b_\alpha\|_\infty, \lambda, 
\Omega, N, p, m$ and the modulus of continuity of $a_{ij}$ such that  
$$
\|u\|_{W^{2m,p}(\Omega)} \leq C(\|g\|_{L^p(\Omega)}+ \|u\|_{L^p(\Omega)}).
$$
\label{adn_global}
\end{theorem} 

We will also be using the following local analogue of this result. 
Though the proof may be standard we give it for the reader's convenience.

\begin{corollary} 
Let $\Omega$ be a ball $\{x\in \R^N: |x|<R\}$ or a half-ball 
$\{x\in \R^N: |x|<R, x_1 >0\}$. Let $m\in \N$, 
$a_{ij}\in C^{2m-2}(\overline{\Omega})$, $b_\alpha \in L^\infty(\Omega)$, 
$g\in L^p(\Omega)$ for some $p\in (1,\infty)$. Suppose 
$u\in W^{2m,p}(\Omega)$ satisfies \eqref{linear_eq} 
\begin{itemize}
\item[(i)] either on the ball 
\item[(ii)] or on the half-ball together with the boundary conditions 
$u=\frac{\partial}{\partial x_1}u=\ldots=
\frac{\partial^{m-1}}{\partial x_1^{m-1}} u =0$ on 
$\{x\in\R^N: |x|<R, x_1=0\}$.
\end{itemize}
Then there exists a constant $C>0$ depending only on 
$\|a_{ij}\|_{C^{2m-2}}, \|b_\alpha\|_\infty, \lambda, \Omega, N, p, m$, the 
modulus of continuity of $a_{ij}$ and $R$ such that for any $\sigma\in (0,1)$ 
$$
\|u\|_{W^{2m,p}(\Omega\cap B_{\sigma R})} \leq \frac{C}{(1-\sigma)^{2m}}
(\|g\|_{L^p(\Omega)}+ \|u\|_{L^p(\Omega)}).
$$
\label{adn_local}
\end{corollary}

\begin{proof} It is sufficient to prove the result for $R=1$. For 
$\sigma\in (0,1)$ let $\eta\in C_0^{2m}(B_1)$ be a cut-off 
function with $0\leq \eta \leq 1$, $\eta\equiv 1$ in $B_{\sigma}$, 
$\eta\equiv 0$ for $|x|\geq \sigma'$ where $\sigma'=\frac{1+\sigma}{2}$ and 
$|D^\gamma\eta|\leq \Big(\frac{4}{1-\sigma}\Big)^{|\gamma|}$ for 
$|\gamma|\leq 2m$. Then 
$$
L(u\eta) = g\eta + \sum_{\substack{|\beta|\leq 2m-1 \\ 
|\gamma|\leq 2m-|\beta|}} 
c_{\beta,\gamma}(x) D^\beta u D^\gamma \eta \mbox{ in }\Omega,
$$
where $c_{\beta,\gamma}$ are 
bounded functions with $\|c_{\beta,\gamma}\|_\infty\leq C_1$ and $C_1=C(m)
\max\{\|b_\alpha\|_\infty, \|a_{ij}\|_{C^{2m-2}}\}$. By  Theorem \ref{adn_global} 
\begin{align*}
\|\nabla^{2m} u\|_{L^p(\Omega\cap B_{\sigma})} &
\leq C_2\Big(\|g\|_{L^p(\Omega)}+
\sum_{\substack{ 0\leq k\leq 2m-1 \\ 0\leq l\leq 2m-k}} 
\|\nabla^k u\|_{L^p(\Omega\cap B_{\sigma'})}
(1-\sigma)^{-l}\Big) \\
& \leq C_3\Big(\|g\|_{L^p(\Omega)}+ \sum_{k=0}^{2m-1} 
\|\nabla^k u\|_{L^p(\Omega\cap B_{\sigma'})} (1-\sigma)^{k-2m}\Big).
\end{align*}
If we introduce for $k\in \N_0$ the weighted norm $\Phi_k = 
\sup_{0<\sigma<1} (1-\sigma)^k\|\nabla^k u\|_{L^p(\Omega\cap B_\sigma)}$ 
then the last inequality implies 
\begin{equation}
\Phi_{2m} \leq C_3\Big(\|g\|_{L^p(\Omega)}+\sum_{k=0}^{2m-1} \Phi_k\Big).
\label{Phi}
\end{equation}
Recall the standard interpolation inequality, see Adams, Fournier \cite{AF}, 
for $0\leq k\leq 2m-1$
$$
\|\nabla^k u\|_{L^p(\Omega\cap B_\sigma)}\leq 
\epsilon \|\nabla^{2m} u\|_{L^p(\Omega\cap B_\sigma)}+ 
C \epsilon^\frac{-k}{2m-k} \|u\|_{L^p(\Omega\cap B_\sigma)},
$$ 
where $C$ is homothety invariant and hence independent of $\sigma$.
Using this we find that for every fixed 
$\delta>0$ there exists $\sigma(\delta)\in (0,1)$ such that 
\begin{align*}
\Phi_k & \leq (1-\sigma)^k\|\nabla^k u\|_{L^p(\Omega\cap B_\sigma)}+\delta \\ 
& \leq (1-\sigma)^k\Big((1-\sigma)^{2m-k}\epsilon 
\|\nabla^{2m} u\|_{L^p(\Omega\cap B_\sigma)} + C\epsilon^\frac{-k}{2m-k}
(1-\sigma)^{-k} \|u\|_{L^p(\Omega\cap B_\sigma)}\Big)+\delta\\
& = \epsilon (1-\sigma)^{2m} \|\nabla^{2m} u\|_{L^p(\Omega\cap B_\sigma)} + 
C \epsilon^\frac{-k}{2m-k} \|u\|_{L^p(\Omega\cap B_\sigma)}+\delta.
\end{align*}
Since $\delta>0$ was arbitrary, we see that $\Phi_k \leq \epsilon\Phi_{2m}+
C_\epsilon \Phi_0$. Hence it follows from \eqref{Phi} that 
$\Phi_{2m} \leq C_4(\|g\|_{L^p(\Omega)}+\|u\|_{L^p(\Omega)})$, i.e., 
$$
\|\nabla^{2m}u\|_{L^p(\Omega\cap B_\sigma)} \leq \frac{C_4}{(1-\sigma)^{2m}}
(\|g\|_{L^p(\Omega)}+ \|u\|_{L^p(\Omega)}).
$$
Using the interpolation inequality again we obtain the claim.
\end{proof}

\noindent
{\em Proof of Theorem \ref{main}.} It is convenient to rewrite the operator $L$ in the form 
$$
L=(-1)^m \sum_{|\alpha|=2m} a_\alpha(x)D^\alpha + \sum_{|\alpha|\leq 2m-1} c_\alpha(x)D^\alpha.
$$
Here $a_\alpha(x)= \sum \limits_{I \in \cM_\alpha}a_{i_1 i_2}(x)\cdot 
a_{i_3 i_4}(x)\cdots 
a_{i_{2m-1} i_{2m}}(x)$, where $\cM_\alpha$ is the set of all vectors 
$I=(i_1,\dots,i_{2m}) \in \{1,\dots,N\}^{2m}$ satisfying $\#
\{j\::\:i_j=l\}= \alpha_l$ for $l=1,\dots,N$. Hence $a_\alpha$ is continuous on $\overline{\Omega}$ and 
$c_\alpha\in L^\infty(\Omega)$. Assume for 
contradiction that there exists a sequence $u_k$ of solutions of \eqref{basic} 
with $M_k := \|u_k\|_\infty \to \infty$ as $k\to \infty$. By considering a 
suitable subsequence we can assume that there exists $x_k\in \Omega$ 
such that either $M_k=u_k(x_k)$ for all $k\in\N$ (positive blow-up) or 
$M_k = -u_k(x_k)$ for all $k\in \N$ (negative blow-up). Define 
$$
v_k(y) := \frac{1}{M_k} u_k(M_k^\frac{1-q}{2m} y+x_k).
$$
Then $\|v_k\|_\infty=1$ and either $v_k(0)=1$ for all $k\in\N$ 
(positive blow-up) or $v_k(0)=-1$ for all $k\in\N$ (negative blow-up). We may 
also assume that $x_k\to \bar x\in \overline{\Omega}$. 

\smallskip

\noindent
\underline{Case 1:} $\bar x\in \Omega$. In this case $v_k$ is well-defined on 
the sequence of balls $B_{\rho_k}(0)$ with $\rho_k:= M_k^\frac{q-1}{2m}
\dist(x_k,\partial\Omega)\to \infty$ as $k\to \infty$. Note that 
$$
D^\alpha v_k(y) = M_k^{\frac{1-q}{2m}|\alpha|-1}
(D^\alpha u_k)(M_k^\frac{1-q}{2m} y+x_k). 
$$
For $y\in B_{\rho_k}(0)$ let
\begin{equation}
\bar a_\alpha^k(y) := a_\alpha(M_k^\frac{1-q}{2m} y+x_k), \quad 
\bar c_\alpha^k(y) := M_k^{(q-1)(\frac{|\alpha|}{2m}-1)} 
c_\alpha(M_k^\frac{1-q}{2m} y+x_k)
\label{def_coeff}
\end{equation}
and define the operator
\begin{equation}
\bar L^k :=(-1)^m \sum_{|\alpha|=2m} \bar a_\alpha^k(y) D^\alpha
+\sum_{|\alpha|\leq 2m-1} \bar c_\alpha^k(y) D^\alpha.
\label{def_op}
\end{equation}
The function $v_k$ satisfies
\begin{equation}
\bar L^k v_k(y) = f_k(y) \mbox{ in } B_{\rho_k}(0), \quad \mbox{ where } 
f_k(y) := \frac{1}{M_k^q} f(M_k^\frac{1-q}{2m} y+x_k, M_k v_k(y)). 
\label{eq_vk}
\end{equation}
By our assumption on the nonlinearity $f(x,s)$ we have that $\|f_k\|_{L^\infty(B_{\rho_k}(0))}$ 
is bounded in $k$. Note that while the ellipticity constant and the 
$L^\infty$-norm of the coefficients of $\bar L^k$ are the same as for $L$, 
the modulus of continuity of $\bar a_\alpha^k$ is smaller than that of $a_\alpha$. 
By applying Corollary~\ref{adn_local} on the ball 
$B_R(0)$ for any $R>0$ and any $p\geq 1$ there exists a constant $C_{p,R}>0$ 
such that 
$$
\|v_k\|_{W^{2m,p}(B_R(0))} \leq C_{p,R} \mbox{ uniformly in } k.
$$
For large enough $p$ we may extract a subsequence (again denoted $v_k$) such 
that $v_k \to v$ in $C^{2m-1,\alpha}(B_R(0))$ as $k\to \infty$ for every 
$R>0$, where $v\in C^{2m-1,\alpha}_{loc}(\R^N)$ is bounded with 
$\|v\|_\infty=1=\pm v(0)$. Taking yet another subsequence we may assume that 
$f_k\overset{\ast}{\rightharpoonup} F$ in $L^\infty(K)$ as $k\to \infty$ for every compact set 
$K\subset \R^N$. Also we see that 
\begin{equation}
F(y) = \left\{\begin{array}{ll}
h(\bar x)v(y)^q & \mbox{ if } v(y)>0,\vspace{\jot}\\
k(\bar x)|v(y)|^q & \mbox{ if } v(y)<0,
\end{array} \right.
\label{def_F}
\end{equation}
because, e.g., if $v(y)>0$ then there exists $k_0$ such that $v_k(y)>0$ for $k\geq k_0$ and hence 
$M_kv_k(y)\to \infty$ as $k\to \infty$. Therefore the assumption on $f(x,s)$ implies that 
$f_k(y)\to h(\bar x) v(y)^q$ as $k\to \infty$, and a similar pointwise convergence holds at 
points where $v(y)<0$.  Finally, note that the pointwise convergence of 
$f_k$ on the set $Z^+=\{y\in \R^N: v(y)>0\}$ and $Z^-=\{y\in \R^N:v(y)<0\}$ determine due to 
the dominated convergence theorem the 
weak$\ast$-limit $F$ of $f_k$ on the set $Z^+\cup Z^-$. Since $\bar b_\alpha^k(y)\to 0$ and 
$\bar a^k_\alpha(y)\to a_\alpha(\bar x)$ as $k\to\infty$ and since we may assume that 
$v_k\to v$ in $W^{m,p}_{loc}(\R^N)$ we find that $v$ is a bounded, 
weak $W^{m,p}_{loc}(\R^N)$-solution of 
\begin{equation} 
{\mathcal L} v= F \mbox{ in } \R^N, \quad \mbox{ where} \quad 
{\mathcal L} = (-1)^m \sum_{|\alpha|=2m} a_\alpha(\bar x) D^\alpha
= \Bigl(-\sum_{i,j=1}^N a_{ij}(\bar x)\frac{\partial^2}{\partial y_i\partial y_j} \Bigr)^m. 
\label{eq:whole_space}
\end{equation}
Since $F\in L^\infty(\R^N)$ we get that 
$v\in W^{2m,p}_{loc}(\R^N)\cap C^{2m-1,\alpha}_{loc}(\R^N)$ is a bounded, 
strong solution of \eqref{eq:whole_space}. Because $D^{2m} v = 0 $ a.e. on the set 
$\{y\in \R^N: v(y)=0\}$ we see that $v$ is a strong solution of 
$$
{\mathcal L} v=
\left\{\begin{array}{ll}
h(\bar x)v(y)^q & \mbox{ if } v(y)>0,\vspace{\jot}\\
0 & \mbox{ if } v(y) =0, \vspace{\jot}\\
k(\bar x)|v(y)|^q & \mbox{ if } v(y)<0
\end{array}\right.
$$
in $\R^N$. Notice that the right-hand side of the equation is $C^1(\R^N)$. Hence $v$ is a classical 
$C^{2m,\alpha}_{loc}(\R^N)$ solution. By a linear change of variables we may assume that $v$ solves 
\begin{equation}
(-\Delta)^m v = g(v) \mbox{ in } \R^N, \quad \mbox{ where } 
g(s) = \left\{\begin{array}{ll}
h(\bar x)s^q & \mbox{ if } s \geq  0,\vspace{\jot}\\
k(\bar x)|s|^q & \mbox{ if } s\leq 0.
\end{array}\right.
\label{def_g}
\end{equation}
By Lemma \ref{averages} of the Appendix we find that $v\geq 0$. This already excludes negative 
blow-up and implies that $g(v(y))=h(\bar x)v(y)^q$, $v(0)=1$. Theorem \ref{rn} tells us that this is 
impossible. This finishes the contradiction argument in the first case.

\medskip

\noindent
\underline{Case 2:} $\bar x\in \partial \Omega$. By flattening the boundary through a local 
change of coordinates we may assume that near $\bar x=0$ the boundary is contained 
in the hyperplane $x_1=0$, and that $x_1>0$ corresponds to points inside $\Omega$. Since 
$\partial\Omega$ is locally a $C^{2m}$-manifold, this change of coordinates transforms the 
operator $L$ into a similar operator which satisfies the same hypotheses as $L$. For simplicity we 
call the transformed variables $x$ and the transformed operator $L$. Now the 
function $v_k$ is well-defined on the set 
$B_{\rho_k}(0)\cap\{y\in \R^N: y_1> -M_k^\frac{q-1}{2m} x_{k,1}\}$. Since 
$$
1 = |\underbrace{v_k(0)}_{=\pm 1}-\underbrace{v_k(-M_k^\frac{q-1}{2m}x_{k,1},0,\ldots,0)}_{=0}| 
\leq M_k^\frac{q-1}{2m}x_{k,1} \|\nabla v_k\|_\infty
$$
we see that either $M_k^\frac{q-1}{2m}x_{k,1}$ is unbounded and we can conclude as in Case 1, 
or (by extracting a subsequence) $\tau_k := M_k^\frac{q-1}{2m}x_{k,1}\to \tau>0$ as $k\to \infty$. 
In this case we make a further change of coordinates and define 
\begin{align*}
w_k(z) &:= v_k(z_1-\tau_k,z_2,\ldots,z_N),\\ 
\tilde a_\alpha^k(z) &:= \bar a_\alpha^k(z_1-\tau_k,z_2,\ldots,z_N),\\
\tilde c_\alpha^k(z) &:= \bar c_\alpha^k(z_1-\tau_k,z_2,\ldots,z_N)
\end{align*}
and likewise the operator $\tilde L^k$. Note that $w_k(\tau_k,0,\ldots,0)=\pm 1$.  
Let $\R^N_+=\{z\in \R^N:z_1>0\}$ and $B_R^+ = B_R(0)\cap \R^N_+$ for $R>0$. For 
$k$ sufficiently large the  coefficients $\tilde a_\alpha^k$, $\tilde c_\alpha^k$ 
and the operator $\tilde L^k$ are well-defined in $B_R^+$. As before 
$w_k$ satisfies 
$$
\tilde L^k w_k(z) = \tilde f_k(z) \mbox{ in } B_R^+, \quad \mbox{ where } 
\tilde f_k(z) := \frac{1}{M_k^q} f(M_k^\frac{1-q}{2m}z+(0,x_{k,2},\ldots,x_{k,n}), M_k w_k(z)). 
$$
together with Dirichlet-boundary conditions on $\{z\in \R^N:|z|<R, z_1=0\}$. Hence we may 
apply Corollary~\ref{adn_local} on the half-ball 
$B_R^+$ for any $R>0$ and find that for any $p\geq 1$ there exists a constant $C_{p,R}>0$ 
such that 
$$
\|w_k\|_{W^{2m,p}(B_R^+)} \leq C_{p,R} \mbox{ uniformly in } k.
$$
As in Case 1 we can extract convergent subsequences $w_k\to w$ in 
$C^{2m-1,\alpha}_{loc}(\overline{\R^N_+})$ and 
$f_k\overset{\ast}{\rightharpoonup} F$ in $L^\infty(\R^N_+)$ as $k\to \infty$, 
where $F\geq 0, \not \equiv 0$ is determined in the same way as in Case 1. This time, $w$ is a bounded, 
strong $W^{2m,p}_{loc}(\R^N_+)\cap 
C^{2m-1,\alpha}_{loc}(\overline{\R^N_+})$-solution of 
$$
{\mathcal L} w = F \mbox{ in } \R^N_+, \qquad 
\frac{\partial}{\partial z_1}w=\ldots=
\frac{\partial^{m-1}}{\partial z_1^{m-1}} w =0 \mbox{ on }\partial \R^N_+
$$
with ${\mathcal L}$ as in \eqref{eq:whole_space}. 
By a linear change of variables we may assume that $w$ solves 
\begin{equation}
(-\Delta)^m w = g(w) \mbox{ in } \R^N_+, \qquad \frac{\partial}{\partial z_1}w=\ldots=
\frac{\partial^{m-1}}{\partial z_1^{m-1}} w =0 \mbox{ on }\partial \R^N_+,
\label{eq_halfspace}
\end{equation}
where $g$ is defined as in \eqref{def_g} of Case 1.
The representation formula of Theorem \ref{green-poiss-form_plus} shows that $w$ is positive 
and that $g(w(z))=h(\bar x)w(z)^q$. Therefore $w$ is a positive, bounded and classical solution 
$C^{2m}$-solution of $(-\Delta)^m w= h(\bar x) w^q$ in $\R^N_+$ with Dirichlet 
boundary conditions and $w(0)=1$. A contradiction is reached by Theorem~\ref{rn_plus}. \qed

\medskip

\noindent
{\em Proof of Remark \ref{remark_to_main}.} Take sequences of solutions $(u_k,\lambda_k)$ such that 
$M_k := \|u_k\|_\infty\to \infty$ as $k\to \infty$, $\lambda_k\geq \lambda_0>0$ and define the 
rescaled functions
$$
v_k(y) := \frac{1}{M_k} u_k\left(\Big(\frac{M_k^{1-q}}{\lambda_k}\Big)^{1/2m} y+x_k\right).
$$
Due to the assumption $\lambda_k\geq \lambda_0>0$ one has that $M_k^{1-q}/\lambda_k\to 0$ as 
$k\to\infty$. Define further 
\begin{align*}
\bar a_\alpha^k(y) &:= a_\alpha\left(\Big(\frac{M_k^{1-q}}{\lambda_k}\Big)^{1/2m} y+x_k\right), \\
\bar c_\alpha^k(y) &:= M_k^{(q-1)(\frac{|\alpha|}{2m}-1)}\lambda_k^{\frac{|\alpha|}{2m}-1} 
c_\alpha\left(\Big(\frac{M_k^{1-q}}{\lambda_k}\Big)^{1/2m} y+x_k\right)
\end{align*}
with the corresponding operator $\bar L^k$. Then $v_k$ satisfies
$$
\bar L^k v_k(y) = f_k(y) \quad \mbox{ where } 
f_k(y) := \frac{1}{M_k^q\lambda_k} f\left(\Big(\frac{M_k^{1-q}}{\lambda_k}\Big)^{1/2m} y+x_k, 
M_k v_k(y)\right).
$$
Note that $\lim_{k\to\infty} f_k(y)= h(\bar x) v(y)^q$ on $Z^+$ and similarly on $Z^-$. The rest of 
the proof is as before.
\qed

%%%%%%%%%%%%%%%%%%%%%%%%%%%%%%%%%%%%%%%%%%%%%%%%%%%%%%%%%%%%%%%%%%%%%%%%%%%

\section{Green representation}
\label{sec:representation}

The main result of this section is Theorem \ref{green-poiss-form_plus}. There we state 
conditions on a function $u$ on the half-space $\R^N_+$ under which the Green representation 
formula 
$$
u(x)= \int_{\R^N_+} G_\infty^+(x,y) (-\Delta)^m u(y)\,dy \mbox{ for all } x \in \R^N_+
$$
holds. Here $G_\infty^+$ is the half-space Green function, see
\eqref{eq:half-space-green-function} below. In the next section, in
the proof of Theorem \ref{rn_plus}, this representation formula will be applied
to solutions of \eqref{eq:liouville-half-space}. 

\medskip

Let us fix some notation. We recall Boggio's celebrated formula \cite{BO} for
the Green function of the operator $(-\Delta)^m$ with Dirichlet boundary
conditions on the unit ball $\B=\{x\in\R^N: |x|<1\}$: 
\begin{align*}
G_1(x,y) &= k_N^m |x-y|^{2m-N} \int_1^{(\psi(x,y)+1)^{1/2}} \frac{(z^2-1)^{m-1}}{z^{N-1}}\,dz \\
&= \frac{k_N^m}{2}|x-y|^{2m-N}\int_0^{\psi(x,y)}
\frac{z^{m-1}}{(z+1)^{N/2}}\,dz \; \mbox{ with } \; 
\psi(x,y)=\frac{(1-|x|^2)(1-|y|^2)}{|x-y|^2}
\end{align*}
for $x,y\in \B$. Here $k_N^m$ is a suitable normalization constant. 
By dilation we find the Green function for the ball $B_R=\{x\in \R^N:|x|<R\}$ 
as follows
\begin{align*}
G_R(x,y) &= R^{2m-N}G_1\left(\frac{x}{R},\frac{y}{R}\right)\\
&=\frac{k_N^m}{2}|x-y|^{2m-N}\int_0^{\psi_R(x,y)}
\frac{z^{m-1}}{(z+1)^{N/2}}\,dz \; \mbox{ with }\;
\psi_R(x,y)=\frac{(R^2-|x|^2)(R^2-|y|^2)}{R^2|x-y|^2}.
\end{align*}
Next we set $P_R:=(R,0,\dots,0) \in \R^N_+$ and we denote by  
$B_R^+:= \{x \in \R^N\::\: |x-P_R|<R\}$ the ball of radius $R$ shifted by $P_R$. If we let 
$G_R^+$ denote the Green function on $B_R^+$ with respect to
Dirichlet boundary conditions then we find the explicit formula
$$
G_R^+(x,y) = R^{2m-N}G_1\left(\frac{x-P_R}{R},\frac{y-P_R}{R}\right) 
= \frac{k_N^m}{2}|x-y|^{2m-N}\int_0^{\psi^+_R(x,y)}
\frac{z^{m-1}}{(z+1)^{N/2}}\,dz
$$
with 
$$ 
\psi_R^+(x,y) =\frac{(R^2-|x-P_R|^2)(R^2-|y-P_R|^2)}{R^2|x-y|^2}, \qquad x,y\in
B_R. 
$$
Finally, if we let $G^+_\infty$ denote the Green function of the operator
$(-\Delta)^m$ on the half-space $\R^N_+$ subject to Dirichlet boundary
conditions then  
\begin{equation}
  \label{eq:half-space-green-function}
G_\infty^+(x,y) = \frac{k_N^m}{2}|x-y|^{2m-N}\int_0^{\psi_\infty(x,y)}
\frac{z^{m-1}}{(z+1)^{N/2}}\,dz \quad \mbox{ with } \quad \psi_\infty(x,y)
=\frac{4x_1y_1}{|x-y|^2}
\end{equation}
for $x,y\in \R^N_+$.

\begin{lemma}
The Green function $G_R^+$ on $B_R^+$ 
converges pointwise and monotonically to the Green function $G_\infty^+$ on $\R^N_+$. 
\label{monotone}
\end{lemma}

\begin{proof} The pointwise convergence is easily checked. Let $x,y\in B_R^+$. The
  monotonicity of $G_R^+(x,y)$ with respect to $R$ is equivalent to the
  monotonicity of  $\psi_R^+(x,y)$ with respect to $R$. Thus 
\begin{multline*}
\frac{d}{dR} \frac{(R^2-|x-P_R|^2)(R^2-|y-P_R|^2)}{R^2} \\
= -\frac{2}{R^3}
(R^2-|x-P_R|^2)(R^2-|y-P_R|^2)
+\frac{2x_1}{R^2}(R^2-|y-P_R|^2)+\frac{2y_1}{R^2}(R^2-|x-P_R|^2). 
\end{multline*}
Setting $a :=(x-P_R)/R$, $b:=(y-P_R)/R$ we have $|a|^2,|b|^2\leq 1$ and we
obtain from the previous computation
\begin{multline*}
\frac{d}{dR} \frac{(R^2-|x-P_R|^2)(R^2-|y-P_R|^2)}{R^2} \\
 = 2R\Big(-(1-|a|^2)(1-|b|^2)+(1+a_1)(1-|b|^2)+(1+b_1)(1-|a|^2)\Big)\\
 = R\Big(
(\underbrace{a_1+\frac{1}{2}+\frac{1}{2}|a|^2}_{\frac{1}{2}(|a|^2+2a_1+1)})
(1-|b|^2)+(\underbrace{b_1+\frac{1}{2}+\frac{1}{2}|b|^2}_{\frac{1}{2}
(|b|^2+2b_1+1)})(1-|a|^2)\Big),
\end{multline*}
and clearly $|a|^2+2a_1+1=|(a_1+1,a_2,\ldots,a_N)|^2\geq 0$. This establishes
the proof. 
\end{proof}

In \cite{GrSw}, Lemma 3.4, Grunau and Sweers proved the
following estimates for the polyharmonic Green function $G_1$ on the unit ball
if $|k|\geq m$ and $x\in \B$, $y\in \partial \B$:  
\begin{equation}
|D_y^k G_1(x,y)| \leq C_{k,N,m} |x-y|^{m-N-|k|}(1-|x|)^m
\label{grsw}
\end{equation}
for some constant $C_{k,N,m}>0$. For the Green function $G_R$ on $B_R$ and $G_R^+$ on $B_R^+$ 
the estimate \eqref{grsw} transforms as follows: 
\begin{equation}
|D_y^k G_R(x,y)| \leq C_{k,N,m} |x-y|^{m-N-|k|}(R-|x|)^m
\label{grsw_trans}
\end{equation}
if $x\in B_R, y \in \partial B_R$. Likewise, 
\begin{equation}
|D_y^k G_R^+(x,y)| \leq C_{k,N,m} |x-y|^{m-N-|k|} |x|^m 
\label{grsw_trans_plus}
\end{equation}
if $x=(x_1,0,\ldots,0)\in B_R^+$ with $x_1\in (0,R), y\in \partial B_R^+$.

\begin{lemma} Let $G$ be the Green function of $(-\Delta)^m$ with
  Dirichlet boundary condition on an arbitrary ball $B\subset\R^n$ with
  exterior unit normal $\nu$ on $\partial B$. For
  any function $v\in C^{2m-1}(\overline{B})\cap W^{2m,p}(B)$ with $p>\frac{N}{2m}$ one has 
  the following Poisson-Green representation for $x \in B$: for $m$ even
\begin{align}
v(x)=& \sum_{i=1}^{m/2} \oint_{\partial B}\Bigl( \Delta^{i-1} v(y)
\partial_{\nu_y}\Delta_y^{m-i}G(x,y) - \Delta_y^{m-i} G(x,y)
\partial_{\nu_y}\Delta^{i-1}v(y)\Bigr)\,ds_y \label{eq:green-poisson-even}
\\
&+\int_B G(x,y) (-\Delta)^{m}v(y)\,dy.\nonumber    
\end{align}
and for $m$ odd 
\begin{align}
v(x)=& -\sum_{i=1}^{(m-1)/2} \oint_{\partial B}\Bigl( \Delta^{i-1} v(y)
\partial_{\nu_y}\Delta_y^{m-i}G(x,y) - \Delta_y^{m-i} G(x,y)
\partial_{\nu_y}\Delta^{i-1}v(y) \Bigr)\,ds_y \label{eq:green-poisson-odd}\\
&-\oint_{\partial B
}\Delta^{(m-1)/2}v(y) \partial_{\nu_y}\Delta_y^{(m-1)/2}G(x,y)\,ds_y
+ \int_B G(x,y) (-\Delta)^{m}v(y)\,dy. \nonumber
\end{align}
\end{lemma}

\begin{proof} First assume $v\in C^{2m}(\overline{B})$. Consider the identity
\begin{align*}
\sum_{i=1}^m \divergenz\left(\Delta^{i-1}v \nabla
\Delta^{m-i}G-\Delta^{m-i}G\nabla \Delta^{i-1}v\right)
&= \sum_{i=1}^m
\left(\Delta^{i-1}v\Delta^{m-i+1}G-\Delta^{m-i}G\Delta^i
v\right)\\  
&= v\Delta^m G-G\Delta^m v \mbox{ in } B.
\end{align*}
If we integrate this identity over $B$ and take into account that
$D_y^\alpha G(x,y)=0$ for $|\alpha|\leq m-1$ and $x\in B, y\in
\partial B$ then we obtain the claim. For $v\in C^{2m-1}(\overline{B})\cap W^{2m,p}(B)$ we can 
argue by approximation and Lebesgue's dominated convergence theorem if we take into account that 
$\int_B G(x,y)|h(y)|\,dy \leq \const \|h\|_{L^p(B)}$ provided $h\in L^p(B)$ and $p>\frac{N}{2m}$. 
\end{proof}

\begin{theorem}
\label{green-poiss-form_plus}
Suppose that $u \in C^{2m-1}(\overline{\R^N_+})\cap W^{2m,p}_{loc}(\R^N_+)$, $p>\frac{N}{2m}$ 
is a function with the following properties:
\begin{itemize}
\item[(i)] $u$ and all partial derivatives of $u$ of order less than or equal
to $2m-1$ are bounded,
\item[(ii)] $u$ satisfies Dirichlet boundary conditions on $\partial \R^N_+$,
\item[(iii)] $(-\Delta)^{m}u\in L^p_{loc}(\R^N_+)$ is non-negative in $\R^N_+$.
\end{itemize}
Then 
\begin{equation}
u(x)= \int_{\R^N_+}G^+_\infty(x,y)(-\Delta)^mu(y)\:dy \qquad \text{for every $x
  \in \R^N_+$.}
\label{rep_plus}
\end{equation}
\end{theorem}

\begin{proof}
Let us first consider the case where $m$ is even. It clearly suffices to
prove \eqref{rep_plus} for $x=(x_1,0,\dots,0) \in \R^N_+$ with $x_1>0$
fixed. In the following we consider $R>2x_1$. Then $x\in B_R^+$ , $x_1 \in (0,R)$, and \eqref{grsw_trans}
yields for $y \in \partial B_R^+$ 
and $i \le \frac{m}{2}$ the following estimates: 
\begin{align*}
|\Delta_y^{m-i}G_R^+(x,y)| & \le C_{i,N,m}|x-y|^{-m-N+2i} |x|^m, \\
|\partial_{\nu_y}\Delta^{m-i}_yG_R^+(x,y)| & \le C_{i,N,m}|x-y|^{-m-N+2i-1}
|x|^m.
\end{align*}
Combining this with \eqref{eq:green-poisson-even}, we get
\begin{multline*}
\left|\int_{B_R^+} G_R^+(x,y)(-\Delta)^{m}u(y)\,dy-u(x)\right|\\
\leq C |x|^m 
\sum_{i=1}^{m/2} \oint_{\partial B_R^+}\Bigl(|\Delta^{i-1} 
u(y)| |x-y|^{-m-N+2i-1} +|x-y|^{-m-N+2i}
|\partial_{\nu_y} \Delta^{i-1}u(y)| \Bigr)\,ds_y.
\end{multline*}
Since $|x-y| \ge |x|$ for $y \in \partial B_R^+$, we conclude that
\begin{align}
\left|\int_{B_R^+} G_R^+(x,y)(-\Delta)^{m}u(y)\,dy-u(x)\right| &\le C_x\! \sum_{i=1}^{m/2} 
\oint_{\partial B_R^+}\!|x-y|^{-N}\bigl(|\Delta^{i-1} 
u(y)|+|\partial_{\nu_y}\Delta^{i-1}u(y)| \bigr)\,ds_y \nonumber\\
&= C_x \sum_{i=0}^{m/2-1} \oint_{\partial B_R^+}|x-y|^{-N}\Bigl(|\Delta^{i} 
u(y)|+|\partial_{\nu_y}\Delta^{i}u(y)| \Bigr)\,ds_y.\label{green-repr-extra1}
\end{align}
We claim that, for every $x \in \R^N_+$, 
\begin{equation}
  \label{eq:green-repr-extra}
\sum_{i=0}^{m/2-1} \oint_{\partial B_R^+}|x-y|^{-N}\Bigl(|\Delta^{i} 
u(y)|+|\partial_{\nu_y}\Delta^{i}u(y)| \Bigr)\,ds_y \to 0\qquad
\text{as $R \to \infty$.}  
\end{equation}
For $N=1$ this is obvious since
$$
\oint_{\partial B_R^+}|x-y|^{-N}\Bigl(|\Delta^{i} 
u(y)|+|\partial_{\nu_y}\Delta^{i}u(y)| \Bigr)\,ds_y= |x-2R|^{-1}
\Bigl(|u^{(2i)}(2R)|+|u^{(2i+1)}(2R)|\Bigr)
$$
for $i \le \frac{m}{2}-1$ as a consequence of the boundary
conditions. For $N \ge 2$ and $0 \le a \le b \le 2R$ let us consider the set 
$B(a,b):= \{y \in \partial B_R^+\::\: a \le y_1 \le b\}$. On $B(a,b)$ we have 
$y=(y_1,y')$ with $|y'|=\sqrt{2Ry_1-y_1^2}$. For $N\geq 3$ we parameterize 
$B(a,b)$ by the map 
$$
F:\left\{
\begin{array}{rcl}
(a,b)\times \mS^{N-2} & \to & B(a,b), \vspace{\jot}\\
(y_1,\phi) & \mapsto & \Big(y_1,\sqrt{2Ry_1-y_1^2}
\theta(\phi)\Big), 
\end{array}
\right.
$$
where $\phi=(\phi_2,\ldots,\phi_{N-1})$ and 
$$
\theta(\phi) = 
\left(\begin{array}{r} 
\cos\phi_2 \sin\phi_3\sin\phi_4\ldots \sin\phi_{N-1} \\
\sin\phi_2 \sin\phi_3\sin\phi_4\ldots \sin\phi_{N-1} \\
\cos\phi_3\sin\phi_4\ldots \sin\phi_{N-1} \\
\multicolumn{1}{c}{\vdots}\\
\cos\phi_{N-2}\sin\phi_{N-1} \\
\cos\phi_{N-1}
\end{array}
\right).
$$
Let $DF=(b_1|b_2|\cdots|b_{N-1})$ be the Jacobian Matrix of the map $F$ and 
$\gr(DF)=\det(DF^T\cdot DF)=\det(b_i\cdot b_j)_{i,j=1,\ldots,N-1}$ be the 
Gram determinant of $DF$. Since $b_1=\big(1,\frac{R-y_1}{\sqrt{2Ry_1-y_1^2}}
\theta\big)^T$ and $b_i=\big(0,\sqrt{2Ry_1-y_1^2}
\frac{\partial\theta}{\partial\phi_i}\big)^T$ for $i=2,\ldots,N-1$ we find
$b_1\cdot b_1=\frac{R^2}{2Ry_1-y_1^2}$, $b_1\cdot b_i=0$ for $i=2,\ldots,N-1$. 
Therefore 
$$
\sqrt{\gr(DF)}= \sqrt{b_1\cdot b_1}\cdot (2Ry_1-y_1^2)^\frac{N-2}{2}
|\det(D\theta)| = R (2Ry_1-y_1^2)^\frac{N-3}{2}\cdot|\det(D\theta)|.
$$
Since
$$
\det(D\theta)=(-1)^{N-1}\sin\phi_3(\sin\phi_4)^2\ldots(\sin\phi_{N-1})^{N-3}
$$
we obtain finally 
$$
\sqrt{\gr(DF)}\leq R (2Ry_1-y_1^2)^\frac{N-3}{2}.
$$
Therefore we can write the surface integral as follows (if $N=2$ the 
line integral is parameterized by $(y_1,\pm\sqrt{2Ry_1-y_1^2})^T$)
$$
\oint_{B(a,b)}|x-y|^{-N}ds_y =\left\{\begin{array}{ll}
\displaystyle\int_a^b \int_{\mS^{N-2}} 
|x-F(y_1,\phi)|^{-N} \sqrt{\gr DF}\,d\phi\,dy_1 & 
\mbox{ if } N\geq 3, \vspace{\jot}\\
\displaystyle 2\int_a^b \frac{R(2Ry_1-y_1^2)^{-1/2}}{x_1^2+2Ry_1-2x_1y_1}
\,dy_1 & \mbox{ if } N=2.
\end{array}
\right.
$$
Since $|x-F(y_1,\phi)|^2=x_1^2+2Ry_1-2x_1y_1$ and $b \le 2R$, we can now estimate as follows:
\begin{align*}
\oint_{B(a,b)}|x-y|^{-N}ds_y & \le c_1 \int_a^{2R}\!\!\!
\frac{R(2Ry_1-y_1^2)^{\frac{N-3}{2}}}{
(x_1^2+2Ry_1-2x_1y_1)^{\frac{N}{2}}}\:dy_1 = \frac{c_1}{R^2} \int_a^{2R}\!\!
\frac{(2\frac{y_1}{R}-(\frac{y_1}{R})^2)^{\frac{N-3}{2}}}
{(\frac{x_1^2}{R^2}+ 
2\frac{y_1}{R}(1-\frac{x_1}{R}))^{\frac{N}{2}}}\:dy_1\\
&= \frac{c_1}{R} \int_\frac{a}{R}^2
\frac{(2t-t^2)^{\frac{N-3}{2}}}{(\frac{x_1^2}{R^2}+
2t(1-\frac{x_1}{R}))^{\frac{N}{2}}}\:dt
\le 
\frac{c_1}{R} \int_\frac{a}{R}^2
\frac{t^{\frac{N-3}{2}}(2-t)^\frac{N-3}{2}}{(\frac{x_1^2}{R^2}+t)^\frac{N}{2}}
\:dt \\
&= \frac{c_1}{R} \int_\frac{a}{R}^2
\frac{t^{-\frac{1}{2}}(2-t)^\frac{N-3}{2}}{\frac{x_1^2}{R^2}+t}
\biggl(\frac{t}{\frac{x_1^2}{R^2}+t}\biggr)^{\frac{N-2}{2}}\!\!dt
\le \frac{c_2}{R} \int_\frac{a}{R}^2
\frac{t^{-\frac{1}{2}}(2-t)^{-\frac{1}{2}}}{\frac{x_1^2}{R^2}+t}\:dt
\end{align*}
with $c_2=2^{\frac{N-2}{2}}c_1$.
Here we have also used that $R \ge 2x_1$ and $N \ge 2$. From now on we assume $a \le R$ and
split the remaining integral as follows:
\begin{equation}
\label{n2_00} 
\oint_{B(a,b)}|x-y|^{-N}ds_y \le 
\underbrace{\frac{c_2}{R}\int_\frac{a}{R}^1
\frac{t^{-1/2}(2-t)^{-1/2}}{\frac{x_1^2}{R^2}+t}
\,dt}_{=:I_1}+\underbrace{\frac{c_2}{R}\int_1^2
\frac{t^{-1/2}(2-t)^{-1/2}}{\frac{x_1^2}{R^2}+t}\,dt}_{=:I_2}.
\end{equation}
In the first integral $I_1$ we have $(2-t)\geq 1$ and therefore
$$
I_1 \le \frac{c_2}{R} \int_\frac{a}{R}^1 
\frac{t^{-1/2}}{\frac{x_1^2}{R^2}+t}\,dt.
$$
If $a>0$, we conclude that  
\begin{equation}
I_1 \leq \frac{c_2}{R} \int_\frac{a}{R}^1 t^{-\frac{3}{2}}\,dt \leq 
 \frac{2c_2}{\sqrt{aR}},
\label{n2_apos}
\end{equation}
while for $a=0$ the substitution $z= \frac{R^2}{x_1^2}t$ yields
\begin{equation}
I_1 \le 
\frac{c_2}{R} \int_0^1
\frac{t^{-1/2}}{\frac{x_1^2}{R^2}+t}\,dt = \frac{c_2}{x_1} \int_0^\frac{R^2}{x_1^2}\frac{1}{\sqrt{z}(1+z)}\,dz 
\le \frac{c_2}{x_1} \int_0^\infty \frac{1}{\sqrt{z}(1+z)}\,dz = 
\frac{c_3}{x_1}.
\label{n2_a0}
\end{equation}
For $I_2$ we have 
\begin{equation}
  \label{eq:I2}
I_2 \le \frac{c_2}{R}\int_1^2 (2-t)^{-1/2}\,dt= \frac{c_4}{R}.   
\end{equation}
Collecting the inequalities \eqref{n2_00},
 \eqref{n2_apos}, \eqref{n2_a0} and recalling that $R
 \ge 2x_1$, we obtain 
\begin{equation}
\oint_{B(a,b)}|x-y|^{-N}ds_y \le c_5 \left \{
  \begin{aligned}
   &1/x_1&&\qquad \text{for $a=0$,}\\ 
   &1/\sqrt{aR}&& \qquad \text{for $a>0$.}
  \end{aligned}
\right. 
\label{est_n2}
\end{equation}
Now let $\eps>0$. By the standard mean-value theorem using the Dirichlet boundary
conditions and the boundedness of the derivatives of orders up to $m$, there exists $\delta>0$ 
such that 
$$
\sum_{i=0}^{m/2-1} \Bigl(|\Delta^{i} 
u(y)|+|\partial_{\nu_y}\Delta^{i}u(y)| \Bigr) \le \eps \qquad \text{on $\;\{y
  \in \R^N_+\::\: |y_1|\le \delta\}$.}
$$
Hence we apply \eqref{est_n2} and obtain
\begin{align*}
\sum_{i=0}^{m/2-1}& \oint_{\partial B_R^+}|x-y|^{-N}\Bigl(|\Delta^{i} 
u(y)|+|\partial_{\nu_y}\Delta^{i}u(y)| \Bigr)\,ds_y\\
&\le \eps \oint_{B(0,\delta)}|x-y|^{-N}\:ds_y+
c_6 \oint_{B(\delta,2R)}|x-y|^{-N}\:ds_y\\
&\le \eps \frac{c_5}{x_1}+  \frac{c_5c_6}{\sqrt{\delta R}} \;
\to \; \eps \frac{c_5}{x_1} \qquad \text{as $R \to \infty$.}   
\end{align*}
Since $\eps$ was chosen arbitrarily, we conclude that
\eqref{eq:green-repr-extra} holds.

\smallskip

Using
\eqref{green-repr-extra1},~\eqref{eq:green-repr-extra} and 
Lemma \ref{monotone} together with the monotone convergence theorem we get
$$
u(x)=\lim_{R \to \infty}
\int_{B_R^+}G_R^+(x,y)(-\Delta)^{m}u(y)\:dy= 
\int_{\R^N_+}G^+_\infty(x,y)(-\Delta)^mu(y)\:dy.
$$
This finishes the proof in the case where $m$ is an even integer.
In the case where $m$ is odd only minor modifications are needed. We use
\eqref{eq:green-poisson-odd} instead of \eqref{eq:green-poisson-even}
together with the estimates arising from \eqref{grsw_trans_plus}: for $x\in B_R^+$,
$y \in \partial B_R^+$ and $i \le \frac{m-1}{2},$
$$
|\Delta_y^{m-i}G_R^+(x,y)| \le C_{i,N,m}|x-y|^{-m-N+2i} |x|^m
$$
and for $i \le \frac{m+1}{2}$,
$$
|\partial_{\nu_y}\Delta^{m-i}_yG_R^+(x,y)| \le C_{i,N,m}|x-y|^{-m-N+2i-1} |x|^m.
$$
By essentially the same estimates as before, we now obtain
\begin{multline*}
\left|\int_{B_R^+} G_R^+(x,y)(-\Delta)^{m}u(y)\:dy-u(x)\right| \\
\le C_x \oint_{\partial B_R^+}|x-y|^{-N}
\Bigl( \sum_{i=0}^{(m-1)/2}|\Delta^i u(y)|+
\sum_{i=0}^{(m-3)/2}|\partial_{\nu_y}\Delta^{i}u(y)|\Bigr)\:ds_y.
\end{multline*}
Again, as a consequence of the boundary
conditions, for every $\eps>0$ there exists 
$\delta>0$ such that 
$$
\sum_{i=0}^{(m-1)/2}|\Delta^i u(y)|+
\sum_{i=0}^{(m-3)/2}|\partial_{\nu_y}\Delta^{i}u(y)| \le \eps \qquad \text{on
  $\;\{y \in \R^N_+\::\: |y_1|\le \delta\}$.}
$$
We therefore may conclude as in the case where $m$ is even that 
$$
\oint_{\partial B_R^+}|x-y|^{-N}
\Bigl( \sum_{i=0}^{(m-1)/2}|\Delta^i u(y)|+
\sum_{i=0}^{(m-3)/2}|\partial_{\nu_y}\Delta^{i}u(y))|\Bigr)\:ds_y \to 0 \qquad
\text{as $R \to \infty$.} 
$$
Using again the monotone convergence theorem, we conclude that
$$
u(x)=\lim_{R \to \infty}
\int_{B_R^+}G_R^+(x,y)(-\Delta)^{m}u(y)\:dy=
\int_{\R^N_+}G^+_\infty(x,y)(-\Delta)^mu(y)\:dy.
$$
\end{proof}

%%%%%%%%%%%%%%%%%%%%%%%%%%%%%%%%%%%%%%%%%%%%%%%%%%%%%%%%%%%%%%%%%%%%%%%%%%%

\section{Proof of the Liouville Theorem in the half-space}
\label{sec:proof-theorem}
This section is devoted to the proof of Theorem~\ref{rn_plus}.
Let $m\in \N$ and assume that $q>1$ if $N\leq 2m$ and $1<q<\frac{N+2m}{N-2m}$ if $N>2m$. Let $u$ be a classical non-negative bounded solution of 
$$ 
(-\Delta)^m u = u^q \mbox{ in } \R^N_+, \quad u=
\frac{\partial u}{\partial x_1}=\ldots=
\frac{\partial^{m-1}u}{\partial x_1^{m-1}} =0 \mbox{ on }\partial \R^N_+.
$$
We need to show that $u\equiv 0$. From
Theorem~\ref{green-poiss-form_plus} we know that
\begin{equation}
u(x)= \int_{\R^N_+}G^+_\infty(x,y)u^q(y)\:dy \qquad \text{for every $x
  \in \R^N_+$.}
\label{rep_plus-repeat}
\end{equation}
We consider the conformal diffeomorphism 
$$
\phi: \B \to \R^N_+,\qquad \phi(y)= 2\frac{y+e_1}{|y+e_1|^2}-e_1,
$$
where $e_1=(1,0,\dots,0)$ is the first coordinate vector. The following
formula shows how $G^+_\infty$ is related to the Green function $G_1$ on the unit
ball.

\begin{lemma}
  \label{green-conform}
$G_\infty^+(\phi(x),\phi(y))= \Bigl(\frac{2}{|x+e_1|
  |y+e_1|}\Bigr)^{2m-N} G_1(x,y)\quad$ for all $x,y \in \B$.
\end{lemma}

\begin{proof}
An easy calculation yields 
\begin{equation}
|\phi(x)-\phi(y)|= \frac{2|x-y|}{|x+e_1| |y+e_1|} \qquad \text{for $x,y \in \B$.}
\end{equation}
Considering the functions $\psi(x,y)=\frac{(1-|x|^2)(1-|y|^2)}{|x-y|^2}$
and $\psi_\infty(x,y)
=\frac{4x_1y_1}{|x-y|^2}$ as in Section~\ref{sec:representation}, we
obtain
  \begin{align*}
\psi_\infty(\phi(x),\phi(y))&=\frac{|x+e_1|^2\phi_1(x)
  |y+e_1|^2\phi_1(y)}{|x-y|^2}=\frac{(2x_1+2-|x+e_1|^2)(2y_1+2-|y+e_1|^2)}{|x-y|^2}\\
&=\frac{(1-|x|^2)(1-|y|^2)}{|x-y|^2}= \psi(x,y) \quad \text{for $x,y
  \in \B$.}
  \end{align*}
We conclude that 
\begin{align*}
G_\infty^+(\phi(x),\phi(y))&= \frac{k_N^m}{2}|\phi(x)-\phi(y)|^{2m-N}\int_0^{\psi_\infty(\phi(x),\phi(y))}
\frac{z^{m-1}}{(z+1)^{N/2}}\,dz\\
&=\Bigl(\frac{2}{|x+e_1|
  |y+e_1|}\Bigr)^{2m-N} \frac{k_N^m}{2}|x-y|^{2m-N}\int_0^{\psi(x,y)}
\frac{z^{m-1}}{(z+1)^{N/2}}\,dz\\  
&=\Bigl(\frac{2}{|x+e_1|
  |y+e_1|}\Bigr)^{2m-N}  G_1(x,y)\qquad \text{for all $x,y \in \B$.}
\end{align*}
\end{proof}

\begin{corollary}
\label{sol-conform} 
Define the function $v: \B \to \R$ by 
$$
v(x):=|x+e_1|^{2m-N}u(\phi(x))
$$ 
and the function $h:\B \times [0,\infty) \to [0,\infty)$ by 
$$
h(x,t):=2^{2m}|x+e_1|^{-\alpha} \: t^q 
$$
where $\alpha:=N+2m-q(N-2m)\geq 0$ by assumption on $q$. Then $v$ satisfies
\begin{equation}
  \label{eq:representation}
v(x)= \int_{\B}G_1(x,y)h(y,v(y))\:dy\qquad \text{for all $x \in \B$.}
\end{equation}
\end{corollary}

\begin{proof}
The Jacobian determinant of $\phi$ satisfies
$|J_\phi(y)|=\frac{2^N}{|y+e_1|^{2N}}$ for $y \in \R^N_+$. Therefore we have 
\begin{align*}
u(\phi(x))&=\int_{\R^N_+}  G^+_\infty(\phi(x),y)u^q(y)\:dy=\int_{\B} G^+_\infty(\phi(x),\phi(y)) u^q(\phi(y))|J_\phi(y)|\:dy\\     
&=\int_{\B} \left(\frac{2}{|x+e_1||y+e_1|}\right)^{2m-N}
G_1(x,y) \Big(|y+e_1|^{N-2m}v(y)\Big)^q\frac{2^N}{|y+e_1|^{2N}}\:dy\\     
&=\frac{1}{|x+e_1|^{2m-N}}\int_{\B}G_1(x,y)h(y,v(y))\:dy     
\end{align*}
for all $x \in \B$, so that
$$
v(x)=|x+e_1|^{2m-N}u(\phi(x))=
\int_{\B}G_1(x,y)h(y,v(y))\:dy.
$$     
\end{proof}

\begin{proposition}
  \label{symmetry}
The function $v: \B \to \R$ is axially symmetric with respect to the
$x_1$-axis.
\end{proposition}

\begin{proof}
We assume that $N \ge 2$ and $v \not \equiv 0$, since otherwise the statement is trivial.
The integral representation \eqref{eq:representation} implies
that $v$ is strictly positive in $\B$. Note
that for every $x_0 \in \partial \B \setminus \{-e_1\}$ we have that
$$
\lim_{x \to x_0,\:x\in \B}v(x)=\lim_{x \to x_0,\:x\in \B}|x+e_1|^{2m-N}u(\phi(x))=|x_0+e_1|^{2m-N}
u\Big(2\frac{x_0+e_1}{|x_0+e_1|^2}-e_1\Big)=0
$$  
since $2\frac{x_0+e_1}{|x_0+e_1|^2}-e_1 \in \partial \R^N_+$. Hence
the function
$v$ -- extended trivially on $\partial \B \setminus \{-e_1\}$ -- is continuous in 
$\overline B \setminus \{-e_1\}$. We fix a unit vector $e \in \R^N$ perpendicular to 
$e_1$ (i.e., $|e|=1$ and $e \cdot e_1=0$), and we show that $v$ is symmetric with respect to the
hyperplane $T:=\{x \in \R^N\::\: x \cdot e=0\}$. For this we apply a
moving plane argument based on the integral representation
\eqref{eq:representation} and reflection inequalities derived in
\cite{BGW,FGW} for $G_1$. We need some notation. For $\lambda \ge 0$, we consider the open half-space 
$H_\lambda= \{x \in \R^N\::\: x \cdot e>\lambda\}$ and the reflection 
$x \mapsto x^\lambda:= x-2(x\cdot e-\lambda)e$ at
the hyperplane $\partial H_\lambda$. We also consider the set 
$$
J_{\lambda}:= \{x \in \B\::\: \text{$x \cdot e < \lambda$ and $x^{\lambda} \not \in \B$}\}
$$
which has nonempty interior if $\lambda>0$. 
With these definitions, the inequalities stated in \cite[Lemma 4]{BGW}
(see also \cite[Lemma 3]{FGW} for
the biharmonic case) translate into the following reflection inequalities: 
\begin{equation}
\label{eq:Glambda} \left.
\begin{aligned}
&G_1(x^\lambda,y^\lambda)> G_1(x,y^\lambda)\qquad
\text{and}\\
&G_1(x^\lambda,y^\lambda)-G_1(x,y)> G_1(x,y^\lambda)-
G_1(x^\lambda,y)
\end{aligned}
\quad \right \} \qquad \text{for all $x,y \in H_\lambda \cap \B$} 
\end{equation}
and 
\begin{equation}
  \label{eq:extra-est}
G_1(x^{{\lambda}},y)-G_1(x,y)>0 \qquad \text{for $x \in H_{\lambda}
  \cap \B$, $y \in J_\lambda$.}
\end{equation}
Now \eqref{eq:extra-est}
and the strict positivity of $v$ in $\B$ imply that 
\begin{equation}
  \label{eq:new1}
 \int_{J_{\lambda}}[G_1(x^{{\lambda}},y)-G_1(x,y)]\,
 h(y,v(y))\, dy>0  \qquad \text{for $\lambda>0$ and $x \in
H_{\lambda} \cap \B$.}
\end{equation}
We claim that the following reflection inequality holds for every $\lambda>0$:
$$
\leqno{(\cC_\lambda)}\qquad \qquad 
v(x)\le v(x^\lambda)\qquad \text{for all $x \in H_\lambda \cap \B$.}
$$
We put 
$$
\lambda_*:= \inf \{\lambda>0\::\: \text{$(\cC_{\lambda'})$ holds for all
    $\lambda' \ge \lambda$}\}. 
$$
Then $\lambda_* \le 1$. Using the continuity of $v$ in $\B$, it is easy to see that
$(\cC_{\lambda_*})$ holds. We suppose for contradiction that
$\lambda_*>0$. Since $0< |x+e_1|^{-\alpha}\le 
|x^{\lambda_*}+e_1|^{-\alpha}$ for $x \in H_{\lambda_*} \cap \B$ and
$v$ is positive in $\B$, 
$(\cC_{\lambda_*})$ yields
\begin{equation}
  \label{eq:hcomparison}
h(x,v(x))\le h(x^{\lambda_*},v(x^{\lambda_*}))\quad \text{for  $x
  \in H_{\lambda_*} \cap \B$.}  
\end{equation}
We claim that
\begin{equation}
\label{eq:ux-uxlambd-textf}
v(x)< v(x^{\lambda_*})\qquad \text{for all $x \in H_{\lambda_*} \cap \B$.}   
\end{equation}
Indeed, by using \eqref{eq:Glambda}, \eqref{eq:new1} and \eqref{eq:hcomparison} we have 
\begin{align*}
v&(x^{\lambda_*})-v(x)=
\int_{\B}[G_1(x^{{\lambda_*}},y)-G_1(x,y)]\, h(y,v(y))\, dy= \int_{H_{\lambda_*} \cap \B} \ldots \,dy 
+ \int_{\B \setminus H_{\lambda_*}} \ldots \,dy \\
&=
\int_{H_{\lambda_*} \cap \B}\Bigl([G_1(x^{{\lambda_*}},y)-G_1(x,y)]\,h(y,v(y))+
[G_1(x^{{\lambda_*}},y^{\lambda_*})
-G_1(x,y^{\lambda_*})]\,h(y^{\lambda_*},v(y^{\lambda_*}))\Bigr)\,dy\\ 
& +\int_{J_{\lambda_*}}[G_1(x^{{\lambda_*}},y)-G_1(x,y)]\, h(y,v(y))\, dy\\
& > \int_{H_{\lambda_*} \cap \B}[G_1(x^{{\lambda_*}},y^{\lambda_*})
-G_1(x,y^{\lambda_*})]\,[h(y^{\lambda_*},v(y^{\lambda_*}))-h(y,v(y))]\,dy
\ge 0 \quad \text{for $x \in H_{\lambda_*} \cap \B$.} 
\end{align*}
Hence \eqref{eq:ux-uxlambd-textf} is true.  For $0<\mu\le {\lambda_*}$ we now consider the
difference function 
$$
w_\mu: H_\mu \cap \B \to \R,\qquad w_\mu(x)= v(x^\mu)-v(x)
$$ 
and the set
$$
W_\mu:= \{x \in H_\mu \cap \B \::\:
w_\mu(x)<0\}. 
$$
We note that $W_{\lambda_*}= \varnothing$, and we claim that
\begin{equation}
  \label{eq:measure-zero}
|W_\mu| \to 0\qquad \text{as
  $\mu \to {\lambda_*}$,}  
\end{equation}
where $|\cdot|$ denotes the Lebesgue measure. Indeed, let $\eps \in
(0,\lambda_*)$, and consider the compact set
$$
K:=\{x \in \B \::\: x \cdot e \ge \lambda_*+\eps,\:|x| \le 1 -\eps\} \subset H_{\lambda_*} \cap \B
$$
Then $\inf \limits_{x \in K}w_{\lambda_*}(x)>0$ by \eqref{eq:ux-uxlambd-textf}. The
continuity of $v$ in $\B$ implies that there exists ${\lambda}_1 \in ({\lambda_*}-\eps,{\lambda_*})$ 
such that 
$$
\inf_{x \in K}w_\mu(x)>0\quad \text{for
  $\lambda_1 \le \mu \le {\lambda_*}$.}
$$
Hence $W_\mu \subset (H_\mu \cap \B) \setminus K \subset 
\{x \in \B \::\: \text{$|x_1-{\lambda_*}|\le \eps\:$ or  $\:|x| \ge 1-\eps $}\}$ 
and therefore
$$
|W_\mu| \le 2^N\eps + \eps \omega_{N-1} \qquad \text{for
  $\lambda_1 \le \mu \le {\lambda_*}$,}
$$
where $\omega_{N-1}$ denotes the area of the $N-1$-dimensional unit sphere.
Since $\eps$ was chosen arbitrarily small, \eqref{eq:measure-zero} follows.
Next we note that, for $\mu \ge \frac{\lambda_*}{2}$ and $x \in
W_\mu$, 
\begin{align}
h(x^\mu,v(x^\mu))&-h(x,v(x))=|x^\mu+e_1|^{-\alpha}v^q(x^\mu)-|x+e_1|^{-\alpha}v^q(x)
\ge |x+e_1|^{-\alpha}(v^q(x^\mu)-v^q(x))\nonumber\\
&\ge \Bigl(\frac{\lambda_*}{2}\Bigr)^{-\alpha}\![v^q(x^\mu)-v^q(x)]\ge
\Bigl(\frac{\lambda_*}{2}\Bigr)^{-\alpha}\!\!\!q v^{q-1}(x)[v(x^\mu)-v(x)]
\ge c(\lambda_*)w_\mu(x) \label{eq:h-Lipschitz}  
\end{align}
with $c(\lambda_*)=\Bigl(\frac{\lambda_*}{2}\Bigr)^{-\alpha}q
\Bigl(\:\sup \limits_{
x\in H_{\lambda_*/2}\cap \B}v(x)\Bigr)^{q-1}<\infty$. We also note that 
\begin{equation}
  \label{eq:h-final-est}
h(x^\mu,v(x^\mu))-h(x,v(x))  \ge 0 \qquad \text{for $x \in H_\mu
  \setminus W_\mu$.}  
\end{equation}
From now on we assume that $\frac{\lambda_*}{2} \le \mu \le \lambda_*$. For $x \in W_\mu$ 
we use \eqref{eq:Glambda},~\eqref{eq:new1}, \eqref{eq:h-Lipschitz} and
\eqref{eq:h-final-est} to estimate
\begin{align}
&0>w_\mu(x)=
\int_{\B}[G_1(x^{\mu},y)-G_1(x,y)]\, h(y,v(y))\, dy\nonumber\\
&>
\int_{H_\mu \cap \B}\!\Bigl([G_1(x^{\mu},y)-G_1(x,y)]\,h(y,v(y))+
[G_1(x^{\mu},y^\mu)-G_1(x,y^\mu)]\,h(y^\mu,v(y^\mu))\!\Bigr)dy\nonumber
\\ 
&\ge 
\int_{H_\mu \cap \B}[G_1(x^{\mu},y^\mu)-G_1(x,y^\mu)]\,[h(y^\mu,v(y^\mu))-h(y,v(y))]\,dy
\nonumber \\
&\ge 
\int_{W_\mu}[G_1(x^\mu,y^\mu)-G_1(x,y^\mu)]\,[h(y^\mu,v(y^\mu))-h(y,v(y))]\,dy \nonumber \\
&\ge 
c(\lambda_*)
\int_{W_\mu}[G_1(x^\mu,y^\mu)-G_1(x,y^\mu)]w_\mu(y)\,dy  \nonumber\\
& \ge c_{N,m}\: c(\lambda_*) \int_{W_\mu}|x-y|^{1-N}w_\mu(y)\,dy,\label{eq:W_mu-ineq}
\end{align}
where in the last step we use the estimate
$$
0 < G_1(x,y)\le c_{N,m}|x-y|^{1-N}\qquad \text{for $x,y \in \B$, $x \not=y$}
$$
with some $c_{N,m}>0$, which is easily deduced from the integral representation of $G_1$ in 
Section~\ref{sec:representation}. Next we pick $s>1$ large enough such that 
$1<q:=\frac{1}{\frac{1}{N}+\frac{1}{s}}<s$. Then the 
Hardy-Littlewood-Sobolev inequality (see e.g. \cite[Section 4.3]{lieb-loss})
implies that 
$$
\biggl(\int_{\R^N} \Bigl |\int_{W_\mu}|x-y|^{1-N}w_\mu(y)\,dy \Bigr|^s\:dx \biggr)^{\frac{1}{s}}= 
\Bigl\||\cdot|^{1-N} *
(1_{W_\mu} w_\mu )\Bigr\|_{L^s(\R^N)} \le c_{s,q} \|w_\mu\|_{L^q(W_\mu)}
$$
with a constant $c_{s,q}>0$. Combining this inequality with \eqref{eq:W_mu-ineq}
and H\"older's
inequality, we obtain 
\begin{equation}
  \label{eq:symmetry-final-est}
\|w_\mu \|_{L^s(W_\mu)} \le  c_0 \|w_\mu\|_{L^q(W_\mu)} 
\le c_0 |W_\mu|^{\frac{s-q}{sq}} \|w_\mu\|_{L^s(W_\mu)} \qquad \text{with 
$c_0:=c_{N,m}\, c(\lambda_*)\, c_{s,q}$.}
\end{equation}
Now \eqref{eq:measure-zero} and \eqref{eq:symmetry-final-est} imply that 
$\|w_\mu\|_{L^s(W_\mu)}=0$ if $\mu<{\lambda_*}$ is close enough to
${\lambda_*}$. Hence property $(\cC_\mu)$ holds if $\mu<{\lambda_*}$ is close enough to
${\lambda_*}$, which
contradicts the definition of $\lambda_*$. It follows that
$\lambda_*=0$, thus $(\cC_\lambda)$ holds for all $\lambda>0$, as claimed.
By continuity, we now deduce that 
$$
v(x)\le v(x^0)\qquad \text{for $x \in H_0 \cap \B$.}
$$
Repeating the moving plane procedure for $-e$ in place of
$e$, we get
$$
v(x)\ge v(x^0)\qquad \text{for $x \in H_0 \cap \B$.}
$$
Hence equality holds, i.e., $v$ is symmetric with respect to the
hyperplane $T=\{x \in \R^N\::\: x \cdot e=0\}$ as claimed. Since $e$ was chosen arbitrarily 
with $|e|=1$ and $e \cdot e_1=0$ we
conclude that $v$ is axially symmetric with respect to the $x_1$-axis.    
\end{proof}

\begin{corollary}
\label{one-dimensionality}
The function $u$ only depends on the $x_1$-variable.  
\end{corollary}

\begin{proof}
Since $v$ is axially symmetric with respect to the $x_1$-axis, the
same is true for $u$. Let $z \in \R^{N-1}$ be arbitrary, and consider the function
$$
U_z: \R^N_+ \to \R,\qquad U_z(x_1,x')= u(x_1,x'-z)\quad \text{for
  $(x_1,x') \in [0,\infty) \times \R^{N-1}$.}  
$$
Then $U_z$ satisfies the same assumptions as $u$, so it is also
axially symmetric with respect to the $x_1$-axis.
This readily implies that $u$ only depends on the $x_1$-variable.
\end{proof}

\begin{theorem}
  \label{ode-case}
Let $f:[0,\infty) \to [0,\infty)$ be a continuous function with
$f(0)=0$ and $f(s)>0$ for $s>0$. If $u$ is a classical non-negative bounded 
solution of the one-dimensional problem
$$ 
(-1)^{m} u^{(2m)} = f(u)\quad \text{in $(0,\infty)$},\qquad u(0)=u'(0)=\ldots=
u^{(m-1)}(0) =0
$$
then $u\equiv 0$.
\end{theorem}

\begin{proof}
The differential equation admits a first integral given by 
$$
H= \sum_{i=1}^{m-1}(-1)^i u^{(i)}u^{(2m-i)}+(-1)^m\Bigl(\frac{1}{2}(u^{(m)})^2+F(u)\Bigr),
$$
where $F(s):= \int_0^s f(s)\:ds$. Indeed, we calculate that
\begin{align*}
\frac{d H}{dt}&=\sum_{i=1}^{m-1}(-1)^i\Bigl(u^{(i+1)}u^{(2m-i)}+u^{(i)}u^{(2m-(i-1))}\Bigr) 
+(-1)^m\Bigl(u^{(m)}u^{(m+1)}+f(u)u'\Bigr)\\
&=-u'u^{(2m)}+(-1)^{m-1}u^{(m)}u^{(m+1)}+(-1)^m\Bigl(u^{(m)}u^{(m+1)}+f(u)u'\Bigr)\\
&=u'\Big((-1)^m f(u)-u^{(2m)}\Bigr)=0 \quad \text{in $(0,\infty)$.}
\end{align*}
Suppose for contradiction that $u \not \equiv 0$. Since $u$ has
a Green function representation by Theorem~\ref{green-poiss-form_plus}, we infer that 
$u$ is strictly positive in $(0,\infty)$, so that $u^{(2m)}=(-1)^{(m)}f(u)$ has no zero in
$(0,\infty)$. By the
mean value theorem, this implies that $u^{(j)}$ has at most $2m-j$
zeros in $(0,\infty)$ for $j=0,\dots,2m$. Hence every $u^{(j)}$ is
eventually monotone and has a limit as $t \to \infty$ since
it is bounded. From this it clearly follows that 
$$
u^{(j)}(t) \to 0 \qquad \text{as $t \to \infty$ for $j=1,\dots,2m$},
$$
but then also $u(t) \to 0$ by using the equation and the assumptions
on $f$ again. Since $F(0)=0$, we thus find that $H \equiv 0$ in
$[0,\infty)$. In particular, the boundary conditions yield
$$
0=H(0)=\frac{(-1)^m}{2}[u^{(m)}(0)]^2, 
$$
so that $u^{(m)}(0)=0$. If $m=1$
we conclude $u(0)=u'(0)=0$, so that $u \equiv 0$ by the uniqueness of
the solution of the initial value problem -- contrary to what we have assumed.

If $m>1$, we set $v=-u''$ and find that $v$ is a bounded solution of 
$$ 
(-1)^{m-1} v^{(2(m-1))} = f(u)>0\quad \text{in $(0,\infty)$,} \qquad v(0)=v'(0)=\ldots=
v^{(m-2)}(0) =0.
$$
So $v$ also has a Green function representation by Theorem~\ref{green-poiss-form_plus}, and thus $v>0$ in
$(0,\infty)$. In sum, we have 
$$
u(0)=0\quad \text{and}\quad u>0,u''<0 \quad \text{in $(0,\infty)$,}
$$
which forces $u'(0)>0$ and therefore contradicts the boundary
conditions. The proof is finished.
\end{proof}

\begin{altproof}{Theorem~\ref{rn_plus} (completed)}
By Corollary~\ref{one-dimensionality}, we have $u(x)=\hat u(x_1)$ 
with some function $\hat u:[0,\infty) \to \R$. Since $\hat u$
satisfies the assumptions of Theorem~\ref{ode-case}, we conclude that
$\hat u \equiv 0$ and therefore $u \equiv 0$.   
\end{altproof}

%%%%%%%%%%%%%%%%%%%%%%%%%%%%%%%%%%%%%%%%%%%%%%%%%%%%%%%%%%%%%%%%%%%%%%%%%%%

\section{Appendix}

\begin{lemma} Let $v$ be a classical bounded solution of $(-\Delta)^m v = g(v)$ in $\R^N$. If 
$g:\R\to [0,\infty)$ is convex and non-negative with $g(s)>0$ for $s<0$ then $v\geq 0$. 
\label{averages}
\end{lemma}

\begin{proof} Notice that the boundedness of $v$ implies the boundedness of $D^j v$ for 
$j=1,\ldots,2m$. First we show that $(-\Delta)^l v \geq 0$ in 
$\R^N$ for $l=1,\ldots,m-1$. Assume that there exists $l\in \{1,\ldots,m-1\}$ and 
$x_0\in \R^N$ with $(-\Delta)^l 
v(x_0)<0$ but $(-\Delta)^j v\geq 0$ in $\R^N$ for $j=l+1,\ldots,m$. 
We may assume w.l.o.g. that $x_0=0$. Let $v_l := (-\Delta)^l v$ for $l=1,\ldots, m-1$ and 
set $v_0=v$. Then we have 
$$
-\Delta v_0 = v_1, \quad -\Delta v_1 = v_2, \quad \ldots \quad -\Delta v_{m-1}= g(v_0) \mbox{ in } 
\R^N. 
$$
If we define spherical averages $\bar w(x) = \frac{1}{r^{N-1}\omega_N} \oint_{\partial B_r(0)} w(y) 
\,d\sigma_y$, $r=|x|$ then the radial functions $\bar v_0, \bar v_1,\ldots,\bar v_{m-1}$ satisfy 
$$
-\Delta \bar v_0 = \bar v_1, \quad -\Delta \bar v_1 = \bar v_2, \quad \ldots \quad 
-\Delta \bar v_{m-1} \geq  g(\bar v_0) \mbox{ in } \R^N,
$$
where we have used Jensen's inequality and the convexity of $g$. Since $v_l(0)<0$ we also have 
$\bar v_l(0)<0$. Moreover $(r^{N-1}\bar v_l')'=-r^{N-1} \bar v_{l+1}$ or $\leq -r^{N-1}g(\bar v_0)$, 
which in both cases is non-positive. Since $\bar v_l'(0)=0$ we see that $\bar v_l'(r)\leq 0$ for all 
$r>0$. In particular $ \bar v_l(r)\leq \bar v_l(0)<0$. Next we integrate the equation 
$$
\Delta \bar v_{l-1}= -\bar v_l\geq -\bar v_l(0)>0
$$
and obtain $r^{N-1} \bar v_{l-1}'(r)\geq -\frac{r^N}{N} \bar v_l(0)$, i.e, $\bar v_{l-1}'(r) \geq 
-\frac{r}{N} \bar v_l(0)$. The unboundedness of $\bar v_{l-1}'$ yields a contradiction. 

\smallskip

Finally, we need to show that $v=v_0\geq 0$. Assume that $v_0(x_0)<0$ and w.l.o.g. $x_0=0$. Since 
$\Delta \bar v_0=-\bar v_1\leq 0$ we see that $\bar v_0'(r)\leq 0$ and define  
$\alpha:= \lim_{r\to\infty} \bar v_0(r)<0$. Then $\lim_{r\to \infty} \Delta \bar v_{m-1}(r) 
= - g(\alpha)<0$, i.e., $\Delta \bar v_{m-1}(r) \leq -\frac{1}{2}g(\alpha)<0$ for $r \geq r_0$. 
By integration this leads to $\lim_{r\to \infty}\bar v_{m-1}'(r)=-\infty$, which contradicts the 
boundedness of $\bar v_{m-1}'$. This finishes the proof of the lemma.
\end{proof}

\bigskip

\noindent
{\bf Acknowledgement.} The authors thank Hans-Christoph Grunau (Magdeburg) for 
helpful comments and remarks.

%%%%%%%%%%%%%%%%%%%%%%%%%%%%%%%%%%%%%%%%%%%%%%%%%%%%%%%%%%%%%%%%%%%%%%%%%%%

%%%%%%%%%%%%%%%%%%%%%%%%%%%%%%%%%%%%%%%%%%%%%%%%%%%%%%%%%%%%%%%%%%%%%%%%%%%

\end{document}